\documentclass[a4paper,reqno]{amsart}
\usepackage{t1enc}
\usepackage[margin=1in]{geometry}
\usepackage{ifthen}
\usepackage{xcolor}
\usepackage{graphicx}
\renewcommand{\(}{\left(}
\renewcommand{\)}{\right)}
\newcommand{\pref}[1]{(\ref{#1})}

\newcommand{\R}{\mathbf{R}}
\newcommand{\Z}{\mathbf{Z}}
\newcommand{\pr}{\mathbf{P}}
\newcommand{\E}{\mathbf{E}}

\renewcommand{\H}{\mathbb{H}}

\theoremstyle{plain}
\newtheorem{theorem}{Theorem}

\newtheorem{corollary}{Corollary}
\newtheorem{proposition}{Proposition}

\theoremstyle{definition}

\newtheorem{conjecture}{Conjecture}

\theoremstyle{remark}

\newcommand{\formula}[2][nolabel]
{\ifthenelse{\equal{#1}{nolabel}}
 {\begin{align*} #2 \end{align*}}
 {\ifthenelse{\equal{#1}{}}
  {\begin{align} #2 \end{align}}
  {\begin{align} \label{#1} #2 \end{align}}
 }
}

\begin{document}
\title[Sharp estimates of Green function of hyperbolic Brownian Motion]{Sharp estimates of Green function   of hyperbolic Brownian motion}
\author{Kamil Bogus, Tomasz Byczkowski, Jacek Ma\l{}ecki}

\address{Tomasz Byczkowski\\ Institute of Mathematics of Polish Academy of Sciences \\ ul. {\'S}niadeckich 8\\ 00-956 Warsaw, Poland}
\email{tomasz.byczkowski@gmail.com}

\address{Kamil Bogus, Jacek Ma\l{}ecki\\ Faculty of Fundamental Problems of Technology \\ Department of Mathematics\\ Wroc{\l}aw University of Technology \\ ul.
Wybrze{\.z}e Wyspia{\'n}\-skiego 27 \\ 50-370 Wroc{\l}aw,
Poland}
\email{kamil.bogus@pwr.edu.pl, jacek.malecki@pwr.edu.pl}


\thanks{Tomasz Byczkowski was supported by the National Science Centre grant no. 2011/03/B/ST1/00423. Kamil Bogus and Jacek Ma\l{}ecki were supported by the National Science Centre grant no. 2013/11/D/ST1/02622.}

\begin{abstract} The main objective of the work is to provide sharp two-sided estimates of $\lambda$-Green function of hyperbolic Brownian motion of a half-space. We strongly rely on recent results obtained by K. Bogus and J. Malecki \cite{BogusMalecki:2014}, regarding precise estimates of the Bessel heat kernel of half-lines. 
\end{abstract}

\maketitle


\section{Introduction}
The hyperbolic Brownian motion is the canonical diffusion on the hyperbolic space,  with the half of the Laplace-Beltrami operator as its generator. It is related to many important objects, mainly to the Bessel processes and also to some functionals important in mathematical finance \cite{DonatiMartinYor:1997}, \cite{JakubowskiWiesniewolski:2013b}. 
Investigations of hyperbolic Brownian motion have long traditions; for a comprehensive survey, see e.g.
\cite{MatsumotoYor:2005a} or \cite{BGS:2007}, \cite{BR:2006}, \cite{BMZ:2010}, \cite{JakubowskiWiesniewolski:2013a}, \cite{MaleckiSerafin:2012}, \cite{Serafin:2014}, \cite{Zak:2007}. Last years there was great interest in the process exiting a given domain.
In the  papers \cite{BGS:2007}, \cite{BR:2006} one can find a very useful representation of the Poisson kernel of a half-space, which led to provide its precise asymptotics. It turns out, however,  that the most fundamental object of the potential theory is the $\lambda$-Green function of the domain, $\lambda\geq 0$. Unfortunately,  its properties were not investigated until now in this context.  The main objective of this work is to construct an adequate representation of the 
$\lambda$-Green function $G^{\lambda}$ of hyperbolic Brownian motion of the half-space $\{x=(x_1,\ldots,x_n):\,\, x_n>a\}$, $a>0$, and to provide its precise estimates. 
To achieve this, we employ Bessel processes $BES^{(\nu)}$, $\nu>0$, of negative indices, killed at first hitting time of the point $a$. Recent paper \cite{BogusMalecki:2014} (see also \cite{BogusMalecki:2015}) contains  very precise estimates for the transition density function $p_a^{(-\nu)}(u,x,y)$ of the Bessel process $BES^{(\nu)}$, starting at the point $x>a>0$ and  killed at the first hitting time of a point $a$ by the process. We briefly call the function $p_a^{(-\nu)}(u,x,y)$ {\it{the Bessel heat kernel}}. 

\medskip

Results of the paper are based on two important points:
\begin{itemize}
\item{the representation of $G^{\lambda}$ in terms of the appropriate Bessel heat kernel;}
\item{sharp estimates of the above-mentioned Bessel heat kernel \cite{BogusMalecki:2014}.}
\end{itemize}

\medskip 

We first  introduce some notation. If  $n$ denotes the dimension of the hyperbolic space, $n>2$ and
 $\lambda \geq 0$  we denote $\nu = \sqrt{2\lambda + (\frac{n-1}{2})^2}$. Furthermore, 
$x=(\tilde{x},x_n), \ y=(\tilde{y},y_n)$, $x_n,y_n>a$, where $\tilde{x}=(x_1,\ldots x_{n-1}) \in\R^{n-1}$.
The main results of the paper are the following:
\medskip
\begin{center}
{\bf{Representation of $G^{\lambda}$}}
\end{center}
\begin{equation*} 
 G^{\lambda}(x,y) = \(x_n y_n\)^{\mu-\nu}\, \int_0^\infty \frac{1}{(2\pi u)^{\frac{n-1}{2}}} 
	\exp\left({-\frac{|\tilde{x}-\tilde{y}|^2}{2u}}\right)p_a^{(-\nu)}(u,x_n,y_n)\,du \/,
\end{equation*}
where $p_a^{(-\nu)}(u,x_n,y_n)$ is the appropriate Bessel heat kernel.

\medskip

\begin{center}
{\bf{Estimates of $G^{\lambda}$}}
\end{center}
For every $\mu=\frac{n-1}{2}, \ n>2$, $\lambda\geq 0$ and $a\geq 0$ we have
\begin{equation*}
G^{\lambda}(x,y)\stackrel{\lambda,n}{\approx} \left(\frac{2x_ny_n}{|x-y|^{2}}\right)^{\mu-1/2}\left(1\wedge\frac{2(x_n-a)(y_n-a)}{|x-y|^2}\right)\left(1\wedge\frac{2x_ny_n}{|x-y|^2}\right)^{\nu-1/2}\/,
\end{equation*}
whenever $x_n,y_n>a$. Here $\stackrel{\lambda,n}{\approx}$ means that the ratio of the functions on the left and right-hand side is bounded from below and above by positive constants depending only on $\lambda$ and $n$.
\medskip 

 Throughout the paper we  rely on the representation 
of the joint density of an integral functional $A_t^{(-\nu)} = \int_0^t \exp 2(B_s^{(-\nu)})\,ds$ and the geometric Brownian motion with drift
$\exp(B_t^{(-\nu)})= \exp(B_t - \nu\, t)$, given in \cite{MatsumotoYor:2005a}.  An important point in our construction consists of application of the  Lamperti Representation stating that the last coordinate of the hyperbolic Brownian, that is, the geometric Brownian motion 
$\exp(B_t^{(-\nu)})$,  can be identified with the Bessel process $BES^{(-\nu)}$  with the time changed 
according to the functional $A_t^{(-\nu)}$.

\section{Preliminaries}
In this section we collect some preliminary material. For more information on the modified Bessel functions we refer the Reader to \cite{AbramowitzStegun:1972} and \cite{Erdelyi:1954}. For questions regarding Bessel processes, stochastic differential equations and one-dimensional diffusions we refer to \cite{RevuzYor:2005} and to \cite{IW}.  
\subsection{Modified Bessel Functions} 
Various potential-theoretic objects appearing in the theory of Bessel processes, geometric Brownian motion or hyperbolic potential theory are expressed in terms of modified Bessel functions  $I_{\vartheta}$ and $K_{\vartheta}$. For convenience of the Reader we collect here basic information and properties of these functions used in the sequel.

The \textit{modified Bessel function $I_{\vartheta}$ of the first kind} is defined by (see, e.g. \cite{Erdelyi:1954}, 7.2.2 (12)):
  \begin{equation}
    \label{I_definition}
     I_{\vartheta}(z) = \(\frac{z}{2}\)^\vartheta \sum_{k=0}^{\infty} \left(\frac{z}{2}\right)^{2k}\frac{1}{k!\Gamma(k+\vartheta+1)} \/, \quad    z >0 \/,
  \end{equation}
  where $\vartheta\in \R$.
  The \textit{modified Bessel function of the third kind} is defined by (see \cite{Erdelyi:1954}, 7.2.2 (13) and (36)):
  \begin{eqnarray}
    \label{K_definition1}
    K_\vartheta (z) &=& \frac{\pi}{2\sin(\vartheta\pi)}\left[I_{-\vartheta}(z)-I_\vartheta(z)\right]\/,\quad
    \vartheta \notin \Z \/,\\
    \label{K_definition2}
    K_n (z) &=& \lim_{\vartheta\to n}K_\vartheta(z) = (-1)^n
    \frac{1}{2}\left[\frac{\partial I_{-\vartheta}}{\partial \vartheta} -
    \frac{\partial I_{\vartheta}}{\partial \vartheta}\right]_{\vartheta=n}\/,\quad n\in \Z\/.
  \end{eqnarray}

We recall the asymptotic behavior of $I_\vartheta(z)$ at zero
 \begin{eqnarray}
\label{asympt_I_zero}
  I_\vartheta(z)&\sim& \frac{1}{\Gamma(\vartheta+1)}
  \left(\frac{z}{2}\right)^{\vartheta}\,,
   \quad z\to 0^+, \label{asympt_I_0}
 \end{eqnarray}
 where  $g(r) \sim f(r) $ means that the ratio of $g$ and $f$ tends to $1$, and at infinity
\begin{equation} \label{asympt_I_infty}
I_\vartheta(z)= (2\pi z)^{1/2}\,e^z [1+ O(z^{-1})]\/,\quad z\to \infty\/. 
\end{equation}

\subsection{Bessel processes}
The basic material  concerning  Bessel processes is taken from
\cite{RevuzYor:2005}, Ch. XI. We begin with a definition of \textit{squared Bessel process} $BESQ^{\delta}(x)$ started at $x\geq 0$. It is defined as the unique strong solution of the equation 
\begin{equation}
\label{defBess}
dZ(t) = 2\,\sqrt{|Z(t)|}d\beta(t) + \delta\,dt\,, \quad Z(0)=x\,,
\end{equation}
where $\beta(t)$ denotes a one-dimensional Brownian motion. Here $\delta\in \R$ is called {\it{the dimension}}
 of $BESQ^{\delta}$. 
For $\delta\geq 0$ the process $Z(t)$ is non-negative and the square root in the equation \pref{defBess} can be omitted. 
It is known that for $0<\delta<2$ the point $0$ is {\it{reflecting}}; for $\delta=0$ it is {\it{absorbing}}. When  $\delta < 0$ the situation gets more trickier:
the process $Z(t)$, when starting from $x>0$ attains the point $0$ in finite time and, after that, becomes negative (and behaves like
the process $(-Z(t))$ with positive dimension $(-\delta)$ (see  \cite{GoingYor:2003}). In this paper, however, we impose the killing
condition at the point $0$ and again call the resulting process {\it{Bessel process}} (of negative dimension). Thus, in our setting, the process $Z(t)$ is always non-negative so we are able to take the square root.

The square root of $BESQ^{\delta}(x^2)$, $x\geq 0$, is called the {\it Bessel process} of dimension $\delta$ started at $x$ and is denoted by ($BES^{\delta}(x)$). We introduce also the {\it index} $\mu=(\delta/2)-1$ of the corresponding process, and write $BES^{(\mu)}$ (instead o {$BES^{\delta}$}) if we want to use $\mu$ (instead of $\delta$). 

For $\mu\geq 0$ the probability density function of the $BES^{(\mu)}(x)$ semigroup is of the form
\begin{equation}
\label{BES}
p^{(\mu)}(t,x,y) = \frac{y}{t}
\(\frac{y}{x}\)^{\mu}
  e^{-(x^2+y^2)/2t}I_{\mu} 
     \(xy/{t}\)
\quad {\textrm{for}} \quad x>0
\end{equation}
and
\begin{equation}
\label{BESzero}
p^{(\mu)}(t,0,y) = 2 (2t)^{-\mu-1}t^{-(\mu+1)}\Gamma(\mu+1)^{-1}\,y^{2\mu+1}
  e^{-y^2/2t} \,.
\end{equation}
Note that the above-given formula describes the density considered with respect to Lebesgue measure. However, sometimes the symmetric version of the density can be more convenient. Then we have to deal with the reference measure $m^{(\mu)}(dy) = m^{(\mu)}(y)dy$, where $m^{(\mu)}(y)=y^{2\mu+1}$. We will switch to the symmetric version in Section \ref{sec:GreenFunction}.

 We denote by $\textbf{P}_x^{(\mu)}$ and $\E_x^{(\mu)}$ the probability law and the corresponding expected value of a Bessel process $BES^{(\mu)}(x)$ on the canonical paths space with the starting point $x\geq 0$. Let $\mathcal{F}_t=\sigma\{R_s,s\leq t\}$ be the filtration of the coordinate process $R_t$. Let us denote the first hitting time of the level $a\geq 0$ by
\begin{eqnarray*}
   T_a = \inf\{t>0; R_t=a\}\/.
\end{eqnarray*}
We have the absolute continuity property for the laws of the Bessel processes with different indices
\begin{eqnarray}
  \label{Bessel:AC}
  \left. \frac{d\textbf{P}^{(\mu)}_x}{d\textbf{P}^{(\nu)}_x}\right|_{\mathcal{F}_t} = \left(\frac{R_t}{x}\right)^{\mu-\nu}\exp\left(-\frac{\mu^2-\nu^2}{2}\int_0^t\frac{ds}{{R_s^2}}\right)\/,\quad {\textbf{P}^{(\nu)}_x}\textrm{ - a.s. on }\{T_0>t\}\/.
\end{eqnarray}
If $\nu\geq 0$ then the condition $\{T_0>t\}$ can be omitted. 
In this work we are interested primarily in the case of negative drift, which we denote by writing the index as $(-\mu)$, thus assuming in the sequel that
$\mu\geq 0$. Observe that the formula \pref{Bessel:AC} implies that in the case of $(-\mu)$  the density function of $BES^{(-\mu)}(x)$, $x>0$ is still of the form \pref{BES} with $\nu=-\mu$. 

\subsection{Exponential functionals}

We now begin with a brief description of Matsumoto-Yor approach \cite{MatsumotoYor:2005a}, which is a starting point of our construction. The important point consists of the computation of the functional
\begin{equation*}
\phi(y)=\E_x^{(0)}\left[\exp\left(-\frac{\mu^2}{2}\int_0^t\frac{ds}{R_s^2}\right)|R_t=y\right]\,.
\end{equation*}
To do this we apply the formula \pref{Bessel:AC} to obtain
\begin{eqnarray*}
&{}& \E_x^{(0)}\left[e^{-\alpha R_t}\,\left(\frac{R_t}{x}\right)^{-\mu}\,
\E_x^{(0)}\left[\exp \left(-\frac{\mu^2}{2}\int_0^t\frac{ds}{R_s^2}\right)|R_t\right]\right] \\
&=& \int_0^{\infty} e^{-\alpha\,y}\,\frac{y}{t}\,\(\frac{y}{x}\)^{-\mu}\,\exp\(-\frac{x^2+y^2}{2t}\)\,I_0 \(\frac{xy}{t}\)\,\phi(y)\,dy\\
&=& \E_x^{(-\mu)}[e^{-\alpha R_t^{(-\mu)}}] =
\int_0^{\infty} e^{-\alpha\,y}\,\frac{y}{t}\,\(\frac{y}{x}\)^{-\mu}\,\exp\(-\frac{x^2+y^2}{2t}\)\,I_{|\mu|} \(\frac{xy}{t}\)\,dy\/.
\end{eqnarray*}
The above equality provides the desired formula
\begin{equation*}
\phi(y)=\E_x^{(0)}\left[\exp\left(-\frac{\mu^2}{2}\int_0^t\frac{ds}{(R_s)^2}\right)|R_t=y\right]
= \frac{I_{|\mu|} \left(\frac{xy}{t}\right)}{I_0 \left(\frac{xy}{t}\right)}\/.
\end{equation*}
Observe that the above formula defines the so-called Hartman-Watson distribution $\eta_r$ with the parameter $r=xy/t$
described in the terms of Laplace transform:
\begin{equation*}
\int_0^{\infty} e^{-\mu^2\,t/2}\,\eta_r(dt) = \frac{I_{|\mu|}(r)}{I_{0}(r)}\/.
\end{equation*}
Below, we need a related density function, denoted by $\theta(r,t)$, defined as
\begin{equation} \label{HW}
\int_0^{\infty} e^{-\mu^2\,t/2}\,\theta(r,t)\,dt = I_{|\mu|}(r) \quad r>0\/.
\end{equation}

\section{Joint distribution of $(A_t^{(-\mu)},B_t^{(-\mu)})$}
Let $(B_t)_{t\geq 0}$ be the Brownian motion in $\mathbb{R}$ starting from $x$. We denote by $\pr^{x}$ and $\E^x$ the corresponding probability law and expected value. Note that the starting point is now indicated in the superscript to distinguish it from the probability law and the expected value for the Bessel process with index $\mu\in\R$ introduced previously. For $\mu\geq 0$ we denote by $B^{(-\mu)}_t=B_t-\mu t$ the Brownian motion with negative drift. We define
\begin{equation*}
  A^{(-\mu)}_t = \int_0^t \exp(2B_s^{(-\mu)})ds\/, \quad t\geq 0\/.
\end{equation*}
For $\mu=0$ we will use the shortened notation $A_t:=A^{(0)}_t$. 

To find the joint distribution of $(A_t^{(-\mu)},B_t^{(-\mu)})$ we proceed as follows.
The crucial point is to determine the formula for the Green function $G(x,y,\alpha^2/2)=(H_{\lambda}+\alpha^2/2)^{-1}$, which is defined as a formal inverse, with $H_{\lambda}$ being the 
following Schrodinger operator
\begin{equation} \label{Schr}
H_{\lambda} = -\frac{1}{2} \frac{d^2}{dx^2} + \frac{1}{2} \lambda^2 e^{2x}\,.
\end{equation}
We first assume that $\mu=0$. For $x\leq y$ we obtain 
\begin{equation} \label{GreenSchr}
G(x,y,\alpha^2/2)=2\,I_{\alpha}(\lambda e^x)\,K_{\alpha}(\lambda e^y)\/.
\end{equation}
It is easily seen that the Feynmann-Kac semigroup generated by the operator \pref{Schr} is of the form
\begin{eqnarray*} 
e^{-H_{\lambda} t}\,f(x) &=& \E^x[e^{-\lambda^2/2\,\int_0^t \exp(2B_s)\,ds}\,f(B_t)] \\
&=& \E^x[\E^x[e^{-\lambda^2/2\,\int_0^t \exp(2B_s)\,ds}|f(B_t)]]\\
&=&\int_0^{\infty}\E^0[e^{-\lambda^2/2\,\,e^{2x}\,A_t}|B_t=y-x]\,f(y)\,\frac{e^{-(y-x)^2/2t}}{\sqrt{2\pi\,t}}\,dy\/.
\end{eqnarray*}
This provides the following formula for the transition density of the above semigroup
\begin{eqnarray*} 
 g_{\lambda}(t;x,y)&=& \E^0[e^{-\lambda^2/2\,\,e^{2x}\,A_t}|B_t=y-x]\,\frac{e^{-(y-x)^2/2t}}{\sqrt{2\pi\,t}}\\
 &=& \int_0^{\infty} e^{-\lambda^2/2\,\,e^{2x} u} \pr^0(A_t\in du|B_t=y-x)\,\frac{e^{-(y-x)^2/2t}}{\sqrt{2\pi\,t}}\/.
\end{eqnarray*}
Integrating with respect to $t>0$ we find that
\begin{equation*} 
 \int_0^{\infty} e^{-\alpha^2\,t/2}\,g_{\lambda}(t;x,y)\,dt = G(x,y,\alpha^2/2) = 2\,I_{\alpha}(\lambda e^x)\,K_{\alpha}(\lambda e^y)\/.
\end{equation*}
Taking into account the product formula for Bessel functions and the definition of the function $\theta(r,t)$ we obtain
for $x=0$
\begin{eqnarray*} 
G(0,y,\alpha^2/2)  &=& \int_0^{\infty} e^{-u/2}\,e^{-\lambda^2/2u}\,e^{-\lambda^2\,e^{2y}/2u}
\int_0^{\infty} e^{-\alpha^2 t/2}\,\theta\(\frac{\lambda^2 e^y}{u},t\)\,dt\,\frac{du}{u}\\
&=& \int_0^{\infty} e^{-\lambda^2/2\xi}\,e^{-1/2\xi}\,e^{-e^{2y}/2\xi}
\int_0^{\infty} e^{-\alpha^2 t/2}\,\theta\(\frac{ e^y}{\xi},t\)\,dt\,\frac{d\xi}{\xi}\\
&=& \int_0^{\infty} e^{-\alpha^2 t/2}\,\int_0^{\infty} e^{-\lambda^2/2u}\,e^{-(1+e^{2y})/2u}\,
 \theta\(\frac{ e^y}{u},t\)\,dt\frac{du}{u}\\
&=& \int_0^{\infty} e^{-\alpha^2 t/2} \int_0^{\infty}e^{-\lambda^2/2u}\,
 \pr^0(A_t\in du|B_t=y)\,\frac{e^{-y^2/2t}}{\sqrt{2\pi\,t}}dt\/.
\end{eqnarray*}
By comparing, we obtain for $x=0$
\begin{equation}
\label{JointDensity0}
\pr^0(A_t\in du|B_t=y)\,\frac{e^{-y^2/2t}}{\sqrt{2\pi\,t}} =  \exp\(-\frac{1+e^{2y}}{2u}\)\theta\(\frac{e^{y}}{u},t\)\,\frac{du}{u}\/,
\end{equation}
and for arbitrary $x$
\begin{equation}
\label{JointDensityx}
\pr^x(A_t\in du,B_t=y)\,\frac{e^{-(y-x)^2/2t}}{\sqrt{2\pi\,t}} =  \exp\(-\frac{e^{2x}+e^{2y}}{2u}\)\theta\(\frac{e^{x+y}}{u},t\)\,\frac{du}{u}\/;
\end{equation}
which, after taking into account the drift formula (Girsanov) gives the following
 formula for the joint density of the variables $(A_t^{(-\mu)},B_t^{(-\mu)})$
\begin{equation}
\label{JointDensity}
\pr^x(A_t^{(-\mu)}\in du,B_t^{(-\mu)}\in dy) = e^{-\mu^2 t/2}e^{-\mu(y-x)} \exp\(-\frac{e^{2x}+e^{2y}}{2u}\)\theta\(\frac{e^{x+y}}{u},t\)\,\frac{1}{u}\/.
\end{equation}


\section{Green function of $(A^{(-\mu)}, \exp(B^{(-\mu)}))$}
\label{sec:GreenFunctionAB}
We begin this section with providing the formula for the $\lambda$-potential of the two-dimensional process $(A^{(-\mu)}, \exp(B^{(-\mu)}))$ together with its Laplace transform. The result can be found in the literature, but for completeness of the exposure and for convenience of the Reader we included it together with the proof in the following proposition. 
\begin{proposition}
$\lambda$-potential of the process $(A_t^{(-\mu)}, \exp(B_t^{(-\mu)})$ is of the form
\begin{equation}
\label{Qpotential}
 {\mathcal{Q}}_\mu^\lambda(x,y;u) 
 = \frac{1}{y}\(\frac{x}{y}\)^{\mu} \exp\(-\frac{x^2+y^2}{2u}\) I_\nu\(\frac{xy}{u}\)\frac{1}{u}\/=
 \(\frac{x}{y}\)^{\mu-\nu}\,{\mathcal{Q}}_\nu^0(x,y;u) \/,
\end{equation}
where $\nu=\sqrt{2\/\lambda + \mu^2}$. The Laplace transform of ${\mathcal{Q}}_\mu^\lambda$ is of the form
\begin{equation}
 \mathcal{L}{\mathcal{Q}}_\mu^\lambda(x,y;\cdot)(r^2/2) = \frac{2}{y}\(\frac{x}{y}\)^{\mu} I_\nu(r x)K_\nu(r y) =
 2\(xy\)^{\mu} I_\nu(r x)K_\nu(r y)\, m^{(-\mu)}(y)\/.
\end{equation}
$m^{(-\mu)}(y)=y^{-2\mu-1}$ denotes here the density of the reference measure $dm^{(-\mu)}$.
\end{proposition}
\begin{proof}
Taking into account the formula \pref{JointDensity}  and \pref{HW}, the $\lambda$--potential $\tilde{{\mathcal{Q}}}_\mu^\lambda(x,y;u)$ of 
$(A_t^{(-\mu)},B_t^{(-\mu)})$ is given by
\begin{equation}
  \int_0^\infty e^{-\lambda t} \pr^x(A_t^{(-\mu)}\in du,B_t^{(-\mu)}\in dy)dt = e^{-\mu(y-x)}\exp\(-\frac{e^{2x}+e^{2y}}{2u}\)I_\nu\(\frac{e^xe^y}{u}\)\frac{1}{u}\/,
\end{equation}
where $\nu=\sqrt{2\lambda+\mu^2}$.

Using the formula for the product of modified Bessel function [see \cite{Erdelyi:1953:volII}, p. 64 (37)] we can compute the Laplace transform of
$\tilde{{\mathcal{Q}}}_\mu^\lambda(x,y;u)$ as a function of $u$
\begin{eqnarray*}
  \mathcal{L}\tilde{{\mathcal{Q}}}_\mu^\lambda(x,y;\cdot)(r^2/2) &=& \int_0^\infty e^{-r^2 u /2} \int_0^\infty e^{-\lambda t} P^x(A_t^{(-\mu)}\in du,B_t^{(-\mu)}\in dy)dt\,du \\
  &=& \(\frac{e^x}{e^y}\)^{\mu} \int_0^\infty e^{-r^2 u /2}\exp\(-\frac{e^{2x}+e^{2y}}{2u}\)I_\nu\(\frac{e^xe^y}{u}\)\frac{1}{u}\,du\\
  &=& \(\frac{e^x}{e^y}\)^{\mu} \int_0^\infty e^{-\xi /2} \exp\(-\frac{(r e^{x})^2+(r e^{y})^2}{2\xi}\)I_\nu\(\frac{r e^x\,r e^y}{\xi}\)\frac{d\xi}{\xi}\\
&=& 2 \(\frac{e^x}{e^y}\)^{\mu} I_\nu(r e^x)K_\nu(r e^y)\/. 
\end{eqnarray*}
Changing variables, we  compute the $\lambda$-potential ${\mathcal{Q}}_\mu^\lambda(x,y;u)$ of the process $(A_t^{(-\mu)},\exp(B_t^{(-\mu)}))$. 
We have for $\exp(B_t^{(-\mu)})$ starting from $x$:
\begin{eqnarray*}
   \pr^x(A_t^{(-\mu)}\in du,\exp(B_t^{(-\mu)})\in dy) &=& \frac{1}{y} \pr^{\log x}(A_t^{(-\mu)}\in du,B_t^{(-\mu)}\in d\log y)\\
   &=& e^{-\mu^2 t/2} \frac{1}{y}\(\frac{x}{y}\)^{\mu} \exp\(-\frac{x^2+y^2}{2u}\) \theta\(\frac{xy}{u},t\)\frac{1}{u}
\end{eqnarray*}
and
\begin{eqnarray*}
  {\mathcal{Q}}_\mu^\lambda(x,y;u) &=& \int_0^\infty e^{-\lambda t}pr^x(A_t^{(-\mu)}\in du,\exp(B_t^{(-\mu)})\in dy)dt\\
  &=& \frac{1}{y}\(\frac{x}{y}\)^{\mu} \exp\(-\frac{x^2+y^2}{2u}\) I_\nu\(\frac{xy}{u}\)\frac{1}{u}\/.
\end{eqnarray*}
The function ${\mathcal{Q}}_\nu^\lambda(\cdot,\cdot;u)$ is symmetric with respect to the reference measure $dm^{(-\mu)}$ with the density
\begin{equation*}
   m^{(-\mu)}(y) = \frac{1}{y^{2\mu+1}}\/.
\end{equation*}
When $\lambda=0$, we write ${\mathcal{Q}}_\mu(x,y;u)$ instead of ${\mathcal{Q}}_\mu^\lambda(x,y;u)$. Thus, taking into account the previous
formula we have
\begin{equation*}
  {\mathcal{Q}}_\mu^\lambda(x,y;u) =
  \(\frac{x}{y}\)^{\mu-\nu}\,{\mathcal{Q}}_\nu(x,y;u)  \/.
\end{equation*}
The Laplace transform of $u\longrightarrow {\mathcal{Q}}_\mu^\lambda(x,y;u)$ is given by
\begin{eqnarray*}
  \mathcal{L}{\mathcal{Q}}_\mu^\lambda(x,y;\cdot)(r^2/2) &=& \int_0^\infty e^{-r^2 u /2} {\mathcal{Q}}_\mu^\lambda(x,y;u)du\\
  &=& \int_0^\infty e^{-r^2 u /2} \frac{1}{y}\(\frac{x}{y}\)^{\mu} \exp\(-\frac{x^2+y^2}{2u}\) I_\nu\(\frac{xy}{u}\)\frac{du}{u}\\
  &=& \frac{2}{y}\(\frac{x}{y}\)^{\mu} I_\nu(r x)K_\nu(r y)\\
  &=& 2\(xy\)^{\mu} I_\nu(r x)K_\nu(r y)\, m^{(-\mu)}(y)\/.
\end{eqnarray*}
\end{proof}
To obtain formulas for the Green function we have to deal with the hitting time $\tau_a$. It is the first hitting time of a level $a$ of the process $\exp(B_t^{(-\mu)})$ starting from $x>a$. More precisely, we rather need to consider the density function $q_{\mu}^{x,a}(s)$ (where $x$ is the starting point of $\exp(B_t^{(-\mu)})$) of the following stopped integral functional
\begin{equation*}
   A^{(-\mu)}_{\tau_a} = \int_0^{\tau_a} \exp(2B_s^{(-\mu)})ds\/.
\end{equation*}
 We have
\begin{eqnarray*}
q_{\mu}^{x,a}(t) &=& \frac12\, \tilde{q}_\mu^{x,a}(t/2)\/,
\end{eqnarray*}
where $\tilde{q}_\mu^{x,a}(t)$ denotes the density investigated in the paper  \cite{BR:2006}. We have the following scaling properties of the density $q_{\mu}^{x,a}(t)$
\begin{equation*}
q_{\mu}^{x,a}(u)= \frac{1}{a^2}q_{\mu}^{x/a,1}(u/a^2)\/.
\end{equation*}
Considering an appropriate Schr\"{o}dinger equation, we can write the Laplace transform of $A^{(-\mu)}_{\tau_a}$ as follows:
\begin{equation*}
   \E^{x}[\exp(-\frac{r^2}{2}A^{(-\mu)}_{\tau_a})] = \(\frac{x}{a}\)^\mu \frac{K_\mu(rx)}{K_\mu(ra)}
	=\int_0^{\infty}\,\exp(-\frac{r^2}{2}\,s)\/ q_{\mu}^{x}(s)\/ds\/.
\end{equation*}
Moreover, for every $\lambda\geq 0$ we get in a similar way
\begin{equation}
\label{lambda_q}
   \E^{x}[\exp(-\lambda\tau_a)\exp(-\frac{r^2}{2}A^{(-\mu)}_{\tau_a})] = \(\frac{x}{a}\)^\mu \frac{K_\nu(rx)}{K_\nu(ra)}\/,
\end{equation}
where $\nu=\sqrt{2\/\lambda+\mu^2}$. Further on we consider that the point $a$ in the definition of $q_{\mu}^{x,a}(t)$ is fixed so we omit this superscript. Taking into account (\ref{lambda_q}) we obtain for any positive Borel function $f$:
\begin{equation}
\label{lambda_density}
\E^{x}[\exp(-\lambda\tau_a)\,;f(A^{(-\mu)}_{\tau_a})] = \(\frac{x}{a}\)^{\mu-\nu} 
\E^{x}[f(A^{(-\nu)}_{\tau_a})] = \(\frac{x}{a}\)^{\mu-\nu} \int_0^{\infty} f(s)\,q_{\nu}^{x}(s)\,ds\,.
\end{equation}

\begin{theorem}
The $\lambda$-Green function of the of the process $(A_t^{(-\mu)},\exp(B_t^{(-\mu)}))$ killed at the first hitting time of the level $a$ has the following 
form (for $a<x<y$)
\begin{equation*}
{\mathcal{G}}_{\mu}^{\lambda}(x,y;u) = \(\frac{x}{y}\)^{\mu-\nu} [{\mathcal{Q}}_\mu(x,y;u)-{\mathcal{Q}}_\mu\ast q_\nu^{x}(a,y;u)]
\end{equation*}
with the corresponding Laplace transform
\begin{eqnarray*}
   \int_0^\infty e^{-r^2u/2} {\mathcal{G}}_{\mu}^{\lambda}(x,y;u)\,du
 &=& 2\(xy\)^{\mu} \((I_\nu(r x)\/K_\nu(r a)- K_\nu(rx)\/I_\nu(ra)\)\frac{K_\nu(ry)}{K_\nu(ra)})\,m^{(-\mu)}(y)\\
   &=& 2\(xy\)^{\mu}\frac{K_\nu(ry)}{K_\nu(ra)}\,S_{\nu}(rx,ra)\,m^{(-\mu)}(y)\/,
\end{eqnarray*}
where $S_{\nu}(\alpha,\beta)=I_\nu(\alpha)\,K_\nu(\beta)- K_\nu(\alpha)\,I_\nu(\beta)$, $0<\beta<\alpha$.
\end{theorem}

\begin{proof}
We begin with the formula for the $\lambda$-harmonic compensator of the process $(A_t^{(-\mu)},\exp(B_t^{(-\mu)}))$, 
with respect to $\tau_a$
\begin{eqnarray*}
 \E^{x}[\exp(-\lambda\tau_a);\, {\mathcal{Q}}_{\mu}^{\lambda}(\exp(B_{\tau_a}^{(-\mu)}),y;A^{(-\mu)}_{\tau_a}+u)] &=&
\(\frac{a}{y}\)^{\mu-\nu} \E^{x}[\exp(-\lambda\tau_a);\, {\mathcal{Q}}_{\nu}(a,y;A^{(-\mu)}_{\tau_a}+u)]\\
 &=&\(\frac{x}{a}\)^{\mu-\nu} \(\frac{a}{y}\)^{\mu-\nu} \int_0^{\infty} {\mathcal{Q}}_{\nu}(a,y;s+u)\,q_{\nu}^{x}(s)\, ds\\	
	 &=&\(\frac{x}{y}\)^{\mu-\nu}\,{\mathcal{Q}}_{\nu}\ast q_{\nu}^{x}(a,y;u)\/.
\end{eqnarray*}
Next we compute the Laplace transform of the compensator with respect to the variable $u$
\begin{eqnarray*}
  \mathcal{L}{\mathcal{Q}}_{\nu}\ast q_{\nu}^{x}(a,y;\cdot)(r^2/2) &=& \mathcal{L}{\mathcal{Q}}_{\nu}(a,y;\cdot)(r^2/2)\cdot
	\mathcal{L}q_{\nu}^{x}(r^2/2)\\
  &=& \frac{2}{y}\(\frac{a}{y}\)^{\nu} I_{\nu}(r a)K_{\nu}(r y)\cdot \(\frac{x}{a}\)^{\nu} \frac{K_{\nu}(rx)}{K_{\nu}(ra)}\\
  &=& \frac{2}{y}\(\frac{x}{y}\)^{\nu} K_{\nu}(rx)K_{\nu}(ry)\frac{I_{\nu}(ra)}{K_{\nu}(ra)}\\
  &=& 2\/\(xy\)^{\nu} K_{\nu}(rx)K_{\nu}(ry)\frac{I_{\nu}(ra)}{K_{\nu}(ra)}\,m^{(-\mu)}(y)\/.
\end{eqnarray*}
The $\lambda$-Green function of the process $(A_t^{(-\mu)},\exp(B_t^{(-\mu)}))$ killed at the first hitting time of the level $a$ has the following form (for $a<x<y$)
\begin{eqnarray*}
  {\mathcal{G}}_{\mu}^{\lambda}(x,y;u) &=& \(\frac{x}{y}\)^{\mu-\nu}[{\mathcal{Q}}_{\nu}(x,y;u)-{\mathcal{Q}}_{\nu}\ast q_{\nu}^{x}(a,y;u)]\\
  &=& \(\frac{x}{y}\)^{\mu-\nu}\int_0^\infty [{\mathcal{Q}}_{\nu}(x,y;u)-{\mathcal{Q}}_{\nu}(a,y;u-s)\textbf{1}_{[0,u)}(s)]q_{\nu}^{x}(s)ds\/.
\end{eqnarray*}
Its Laplace transform is
\begin{eqnarray}
 \label{Green_laplace_0}  
   \int_0^\infty e^{-r^2u/2} {\mathcal{G}}_{\mu}^{\lambda}(x,y;u)\,du &=& \frac{2}{y}\(\frac{x}{y}\)^{\mu} I_{\nu}(r x)\/K_{\nu}(r y)-
	\frac{2}{y}\(\frac{x}{y}\)^{\mu} K_{\nu}(rx)K_{\nu}(ry)\frac{I_{\nu}(ra)}{K_{\nu}(ra)}\\
   &=& 2\(xy\)^{\mu} \((I_{\nu}(r x)\/K_{\nu}(r a)- K_{\nu}(rx)\/I_{\nu}(ra)\)\frac{K_{\nu}(ry)}{K_{\nu}(ra)})\,m^{(-\mu)}(y)\\
   &=& 2\(xy\)^{\mu}\frac{K_{\nu}(ry)}{K_{\nu}(ra)}\,S_{\nu}(rx,ra)\,m^{(-\mu)}(y)\/,
\end{eqnarray}
with $S_{\nu}(\alpha,\beta)=I_{\nu}(\alpha)\,K_{\nu}(\beta)- K_{\nu}(\alpha)\,I_{\nu}(\beta)$, $0<\beta<\alpha$.
\end{proof}
The above-given result immediately implies the following
\begin{corollary}
\begin{eqnarray*}
  {\mathcal{G}}_{\mu}^{\lambda}(x,y;u) = \(\frac{x}{y}\)^{\mu-\nu}{\mathcal{G}}_{\nu}(x,y;u)\,.
\end{eqnarray*}
\end{corollary}


\section{Green function estimates}
\subsection{Hyperbolic Brownian motion in $\H^n$}
We consider the half-space model of the $n$-dimensional real hyperbolic
space
\formula{
\H^n = \{x=(x_1,\ldots, x_n)\in\R^n: x_n>0\}, \quad n\geq2 \/,
}
with the Riemannian metric 
\formula{
  ds^2 = \frac{dx_1^2+\ldots+dx_{n-1}^2+dx_n^2}{x_n^2} \/.
}
The metric induces the hyperbolic distance on $\H^n$ described by 
\formula[eq:distance:formula]{
    \cosh(d_{\H^n}(x,y)) = 1+\frac{|x-y|^2}{2x_ny_n}\/,\quad x,y\in\H^n
}
and the corresponding canonical (hyperbolic) volume element 
\formula{
dV_n=\frac{dx_1\ldots dx_{n-1}dx_n}{x_n^{n}} \/,
}
where $dx_1\ldots dx_{n-1}dx_n$ denotes the Lebesgue measure in $\R^n$. The Laplace-Beltrami operator takes then the following form
\formula{
  \Delta_{\H^n} = x_n^2\sum_{k=1}^n \dfrac{\partial^2}{\partial x_k^2} - (n-2)x_n \dfrac{\partial}{\partial x_n}\/.
}
We define the hyperbolic Brownian motion (HBM) $X(t)=(X_1(t),\ldots,X_n(t))$ as a diffusion on $\H^n$ with the generator $\frac12 \Delta_{\H^n}$. Moreover, we introduce the hyperbolic Brownian motion with drift, i.e. the diffusion $X^{\nu}(t)=(X_1^{\nu},\ldots,X_n^{\nu}(t))$ on $\H^n$ having the half of the operator
\formula{
  \Delta_{\nu} = x_n^2\sum_{k=1}^n \dfrac{\partial^2}{\partial x_k^2} - (2\nu-1) x_n \dfrac{\partial}{\partial x_n}
}
as its generator, where $\nu> 0$.  Notice that 
\formula{
   \frac12 \Delta_\nu = \frac12 \Delta_{\H^n}-(\nu-(n-1)/2)x_n\dfrac{\partial}{\partial x_n}\/,
 }
and consequently, for $\mu=(n-1)/2$ we have $\Delta_{\mu}=\Delta_{\H^n}$ and we go back to HBM, i.e. $X=X^\mu$. We will fix the notation $\mu=(n-1)/2$ in this section.
\subsection{Representations of the hyperbolic Green function of the horocycle}
Our main objective is to study the properties of the Green function and the $\lambda$-Green function of a half-space (or equivalently the interior of the horocycle), i.e. the set 
\formula{
   D = \{x\in \H^n: x_n>a\}\/,
}
where $a>0$. We denote by $\tau_{\nu}$ the first exit time of $X^{\nu}$ from $D$
\formula{
   \tau_\nu = \inf\{t>0: X^\nu(t)\notin D\} = \inf\{t>0: X_n^\nu(t)=a\}\/.
}
Since the last coordinate of the hyperbolic Brownian motion with drift has the same law as the corresponding geometric Brownian motion, we can easily deduce that $\tau_\nu$ is finite almost surely whenever $a>0$. For every $\lambda\geq 0$ we define the $\lambda$-Green function $G^{\lambda}_\nu(x,y)$ of $D$ for the hyperbolic Brownian motion with drift $\nu$ as the integral kernel of the Green operator
\formula{
   G^{\lambda}_\nu f(x) = \int_0^\infty e^{-\lambda t}\E[t<\tau_a,f(X^\nu(t))]\,dt = \int_{D_a} f(y)G^{\lambda}_\nu(x,y)dV_n(y)\/,
}
for every $x,y\in D$ and any Borel function $f$, which is non-negative or bounded. We do not indicate the value of $a$ in the notation of the corresponding objects, i.e. $a$ is assumed to be fixed positive number. We also recall the reference measure formula $dV_n(y) = y_n^{-n}d\tilde{y}dy_n=m^{(-\mu)}(dy_n)d\tilde{y}$, where $\mu=(n-1)/2$.
%
The $\lambda$-potential operator $U^{\lambda}_\nu$ and the $\lambda$-potential kernel $U^{\lambda}_\nu(x,y)$ are defined by
\formula{
U^{\lambda}_\nu f(x) = \int_0^\infty e^{-\lambda t}\E[f(X^\nu(t))]\,dt = \int_{D_a} f(y)U^{\lambda}_\nu(x,y)dV_n(y)\/,\quad x,y\in \H^n\/.
}
For $\lambda=0$ we obtain the Green function of $D$ and the potential kernel, which will be simply denoted by $G_\nu(x,y)$ and $U_\nu(x,y)$ respectively. Finally, if $\nu=\mu=(n-1)/2$ we will omit the subscript $\nu$ in the notation, i.e. $G^\lambda(x,y)$ and $G(x,y)$ are the $\lambda$-Green function and the Green function of the set $D$ for the hyperbolic Brownian motion.

The hyperbolic Brownian motion with drift can be represented by the classical Brownian motion with time changed by an integral functional of geometric Brownian motion and the geometric Brownian motion itself. More precisely, there exists the standard Brownian motion $B=(B(t))=(\tilde{B}(t),B_n(t))$ in $\R^n$ starting from $x = (\tilde{x},\log x_n)$, $x_n>0$ such that 
\formula{
X^\nu(t) = (\tilde{X}^\nu(t),X_n^\nu(t)) = (\tilde{B}(A^{(-\nu)}_t),\exp(B_n^{(-\nu)}(t)))\/,
}
where $B_n^{(-\nu)}(t) = B_n(t)-\nu t$, $A^{(-\nu)}_t = \int_0^t \exp(2B_n^{(-\nu)}(s))ds$. Since the processes $\tilde{B}(t)$ and the two-dimensional process $(A^{(-\nu)}_t,\exp(B_n^{(-\nu)}(t)))$ are independent and the first exit time $\tau_\nu$ depends only on the last coordinate of the process, the above-given representation implies the relation between hyperbolic Green function of the horocycle and the $\lambda$-Green function of $(A^{(-\mu)}_t,\exp(B_n^{(-\mu)}(t)))$ introduced in Section \ref{sec:GreenFunctionAB}. More precisely, the $\lambda$-potential operator for the hyperbolic Brownian motion $X$ and non-negative function $f$ can be written as
\formula{
  U^{\lambda} f(x) &= \int_0^\infty e^{-\lambda t} \E^x f(X(t))dt\\
     &= \int_0^\infty e^{-\lambda t} \int_0^\infty \int_0^\infty \E^{\tilde{x}}f(\tilde{B}(w),y_n) P^{x_n}(A_t^{(-\mu)}\in dw,\exp(B_n^{(-\mu)}(t))\in dy_n) dt\/.
		}
Since the integrand is non-negative, we can change the order of integration to get
\formula{
		U^{\lambda}f(x)
     &= \int_0^\infty \int_{\R^{n-1}}\int_0^\infty f(\tilde{y},y_n) \frac{\exp({-\frac{|\tilde{x}-\tilde{y}|^2}{2w}})}{(2\pi w)^{\frac{n-1}{2}}}\int_0^\infty e^{-\lambda t}\pr^{x_n}(A_t^{(-\mu)}\in dw,\exp(B_n^{(-\mu)}(t))\in dy_n)dt\,d\tilde{y}\,\\
		&= \int_{\R^{n}} \int_0^\infty f(y)\frac{\exp({-\frac{|\tilde{x}-\tilde{y}|^2}{2w}})}{(2\pi w)^{\frac{n-1}{2}}} {\mathcal{Q}}_\mu^\lambda(x_n,y_n;w)dw\, dy\/.
}
Consequently, using the relation
\formula{
   {\mathcal{Q}}_\mu^\lambda(x_n,y_n;w) = \left(\frac{x_n}{y_n}\right)^{\mu-\nu}{\mathcal{Q}}_\nu(x_n,y_n;w)
}
and the formula for the hyperbolic volume element we obtain the corresponding formula for the $\lambda$-potential kernel
\formula[hyp:potential:formula]{
   U^\lambda (x,y) =  y_n^{2\mu+1}\left(\frac{x_n}{y_n}\right)^{\mu-\nu}\int_0^\infty \frac{\exp({-\frac{|\tilde{x}-\tilde{y}|^2}{2w}})}{(2\pi w)^{\frac{n-1}{2}}}
	{\mathcal{Q}}_\nu(x_n,y_n;w)dw\/.
}
Similar computation gives the formula for the $\lambda$-harmonic compensator 
\begin{eqnarray*}
  &&\E^{x}[e^{-\lambda\/\tau_\nu}\, U^\lambda(X(\tau_\nu),y)] = y_n^{2\mu+1}\left(\frac{a}{y_n}\right)^{\mu-\nu}\E^{x}\left[\int_0^\infty\/e^{-\lambda\/\tau_\nu}\/ 
	\frac{\exp({-\frac{|{\tilde{B}}(A_{\tau_\nu}^{(-\mu)})-\tilde{y}|^2}{2w}})}{(2\pi w)^{\frac{n-1}{2}}} 
	{\mathcal{Q}}_\nu(a,y_n;w)\,dw\right]\\
  &=& y_n^{2\mu+1}\(\frac{x_n}{y_n}\)^{\mu-\nu}\,\E^{\tilde{x}}\left[\int_0^\infty\/\int_0^\infty 
	\frac{\exp({-\frac{|{\tilde{B}}(s)-\tilde{y}|^2}{2w}})}{(2\pi w)^{\frac{n-1}{2}}} 
	{\mathcal{Q}}_\nu(a,y_n;w)\/q_\nu^{x_n}(s)\,dw\,ds\right]\\
&=& y_n^{2\mu+1}\(\frac{x_n}{y_n}\)^{\mu-\nu}\,\int_0^\infty  \/\int_0^\infty \int_{\R^{n-1}}\frac{\exp({-\frac{|\tilde{z}-\tilde{y}|^2}{2w}})}{(2\pi w)^{\frac{n-1}{2}}}\,\frac{\exp({-\frac{|\tilde{x}-\tilde{z}|^2}{2s}})}{(2\pi s)^{\frac{n-1}{2}}}
 {\mathcal{Q}}_\nu(a,y_n;w) \/ q_\nu^{x_n}(s)\,dw\/ds\/dz\\
&=& y_n^{2\mu+1}\(\frac{x_n}{y_n}\)^{\mu-\nu}\,\int_0^\infty  \/\int_0^\infty 
\frac{\exp({-\frac{|\tilde{x}-\tilde{y}|^2}{2(w+s)}})}{(2\pi( w+s))^{\frac{n-1}{2}}}\,
 {\mathcal{Q}}_\nu(a,y_n;w) \/ q_\nu^{x_n}(s)\,dw\/ds\\
&=&y_n^{2\mu+1}\(\frac{x_n}{y_n}\)^{\mu-\nu}\, \/\int_0^\infty 
\frac{\exp({-\frac{|\tilde{x}-\tilde{y}|^2}{2u}})}{(2\pi u)^{\frac{n-1}{2}}}\,
 {\mathcal{Q}}_\nu \ast q_\nu^{x_n}(a,y_n;u)\,du\/.
\end{eqnarray*}
The above-given relations enable us to provide the following representation formulas for  the $\lambda$-Green function $G^{\lambda}(x,y)$ 
of a half-space for the hyperbolic Brownian motion:
\begin{theorem}
The $\lambda$-Green function of $D$ is given by
\begin{equation}
\label{hyperbolicGreenf}
   G^{\lambda}(x,y) = y_n^{2\mu+1}\int_0^\infty \frac{\exp({-\frac{|\tilde{x}-\tilde{y}|^2}{2w}})}{(2\pi w)^{\frac{n-1}{2}}} 
	{\mathcal{G}}_{\mu}^{\lambda}(x_n,y_n;w)\,dw\/,
\end{equation}
where ${\mathcal{G}}_{\mu}^{\lambda}(x_n,y_n;w)$ is the Green function for the process $(A_t^{(-\mu)},\exp(B_t^{(-\mu)}))$.
\end{theorem}
The above-given formula together with the relation between a Bessel process and a geometric Brownian motion established by Lamperti relation lead to the following result, which will be crucial in estimating the $\lambda$-Green function.
\begin{theorem}
\formula[eq:GreenBessel]{
  G^{\lambda}(x,y) = \(x_ny_n\)^{\mu-\nu}\, \int_0^\infty \frac{1}{(2\pi u)^{\frac{n-1}{2}}} 
	\exp\left({-\frac{|\tilde{x}-\tilde{y}|^2}{2u}}\right)p_a^{(-\nu)}(u,x_n,y_n)\,du \/,
}
where $p_a^{(-\nu)}(u,x_n,y_n)$ is the density function (with respect to the speed measure $m^{(-\nu)}(dy_n)=y_n^{-2\nu+1}dy_n$) of the transition probability of the Bessel process $BES^{(-\nu)}(x_n)$ killed at the first hitting time $T_a$.
\end{theorem}
\begin{proof}
From the formula (\ref{hyperbolicGreenf}) and the fact that 
\formula{
   \mathcal{G}_\mu^\lambda(x_n,y_n;w) = \left(\frac{x_n}{y_n}\right)^{\mu-\nu}\mathcal{G}_\nu(x_n,y_n;w)
}
we obtain 
\formula{
  G^{\lambda}(x,y) = \(\frac{x_n}{y_n}   \)^{\mu-\nu}G_\nu(x,y)\,\quad x,y\in D\/,
}
where $\nu=\sqrt{2\lambda+\mu^2}$ and $\mu=(n-1)/2$. To prove our theorem, we use the following fact: after changing variables in the geometric Brownian motion according to the formula $t=\alpha_u$ where $\alpha_u= \inf\{s>0;A_s^{(-\nu)}>u\}$ we obtain that
$\exp(B_{\alpha_u}^{(-\nu)})=BES^{(-\nu)}$ is a Bessel process with index $(-\nu)$, with $A_s^{(-\nu)}$ as before. 
Applying this relation we obtain that $T_a = \inf\{s>0; \exp(B_{\alpha_u}^{(-\nu)}) =a \}$
is the first hitting time of the Bessel process $BES^{(-\nu)}$ of the level $a$ and $\tau_a=\alpha_{T_a}$. Applying these relations, 
we obtain for a positive Borel function $f$: 
\begin{eqnarray*}
G_\nu\/f(x) &=& \E^{x}\int_0^{\infty} {\bf{1}}_{\{t<\tau_a\}}\,f({\tilde{B}}_{A_t^{(-\nu)}}, \exp(B_n^{(-\nu)}(t)))\,dt\\
 &=& \E^{\tilde{x}}\E_{x_n}^{(-\nu)}\int_0^{\infty} {\bf{1}}_{\{u<T_a\}}\,f({\tilde{B}}_{u}, R_u(t))\,\frac{du}{R_u^2}\\
&=& \int_0^{\infty}\int_0^{\infty}\E^{\tilde{x}}[\,f({\tilde{B}}_{u},y_n)]\,p_a^{(-\nu)}(u,x_n,y_n)\,\frac{dy_n}{y_n^2y_n^{2\nu-1}}\,du\\
&=&   \int_{\R^{n}}\,\int_0^{\infty}\,\frac{1}{(2\pi u)^{\frac{n-1}{2}}} 
	\exp\left({-\frac{|\tilde{x}-\tilde{y}|^2}{2u}}\right)\,p_a^{(-\nu)}(u,x_n,y_n)\,y_n^{2\mu-2\nu}\,f({\tilde{y}},y_n)\,du\,dV_n(y)\\
\end{eqnarray*}
Collecting all together provides the result.

\end{proof}
\subsection{Sharp estimates of Green functions}
\label{sec:GreenFunction}
We begin with providing the two-sided uniform estimates of the $\lambda$-potential kernels.
\begin{proposition}
  For $n\geq 3$ we have
  \begin{eqnarray*}
      U^{\lambda}(x,y) \approx  \left(\frac{2x_ny_n}{|x-y|^2}\right)^{\mu-1/2}\left(1\wedge\frac{2x_ny_n}{|x-y|^2}\right)^{\nu+1/2}\/,\quad x,y\in\H^n\/,
  \end{eqnarray*}
  where $\mu=\frac{n-1}{2}$ and $\nu=\sqrt{2\lambda+\mu^2}$. Moreover, if $n=2$ we have
    \begin{eqnarray*}
      U^{\lambda}(x,y) \approx \left(1\wedge\frac{2x_ny_n}{|x-y|^2}\right)^{\nu+1/2}\/,\quad x,y\in\H^2\/.
  \end{eqnarray*}
\end{proposition}
\begin{proof}

  Using the formula (\ref{hyp:potential:formula}) together with the asymptotic description of ${\mathcal{Q}}_\mu^{\lambda}(x_n,y_n;t)$ given in \eqref{asympt_I_zero} and \eqref{asympt_I_infty} and the fact that $\mu=\frac{n-1}{2}$, we get
  \begin{eqnarray*}
     U^{\lambda}(x,y) \approx (x_ny_n)^{\mu-1/2}\int_0^{x_ny_n} \frac{1}{u^{n/2}}e^{-\frac{|x-y|^2}{2u}}du+
     (x_ny_n)^{\mu+\nu}\int_{x_ny_n}^{\infty} \frac{1}{u^{\nu+\mu+1}}e^{-\frac{|x-y|^2}{2u}}du\/.
  \end{eqnarray*}
  Making the substitution $\frac{|x-y|^2}{2u}=s$ in both integrals, we arrive at
  \begin{eqnarray*}
      U^{\lambda}(x,y) \approx \left(\frac{2x_ny_n}{|x-y|^2}\right)^{\mu-1/2} \int_{\frac{|x-y|^2}{2x_ny_n}}^\infty s^{n/2-2}e^{-s}ds
      +\left(\frac{2x_ny_n}{|x-y|^2}\right)^{\mu+\nu}\int_0^{\frac{|x-y|^2}{2x_ny_n}}s^{\mu+\nu-1}e^{-s}ds\/.
  \end{eqnarray*}
  Using the following asymptotic description of incomplete gamma functions (see Lemma 12 in \cite{BMR3:2013})
  \begin{eqnarray*}
     \int_0^b s^{\alpha}e^{-s}ds\approx (1\wedge b)^{\alpha+1}\/,\quad b>0\/,\\
     \int_a^\infty s^{\alpha}e^{-s}ds \approx (a+1)^\alpha e^{-a}\/,\quad a>0\/,
  \end{eqnarray*}
  we obtain that for $\frac{|x-y|^2}{2x_ny_n}\geq 1$ that
  \begin{eqnarray*}
     U^{\lambda}(x,y) \approx \frac{2x_ny_n}{|x-y|^{2}}e^{-\frac{|x-y|^2}{2x_ny_n}}+\left(\frac{2x_ny_n}{|x-y|^2}\right)^{\mu+\nu}\approx \left(\frac{2x_ny_n}{|x-y|^2}\right)^{\mu+\nu}\/.
  \end{eqnarray*}
  Finally, whenever $\frac{|x-y|^2}{2x_ny_n}\leq 1$ we have
  \begin{eqnarray*}
     U^{\lambda}(x,y) \approx \left(\frac{2x_ny_n}{|x-y|^2}\right)^{\mu-1/2}+1\/.
  \end{eqnarray*}
  Note that if $n\geq 3$ then the potential behaves as the first summand and in the case $n=2$ as the other one. This ends the proof.
\end{proof}

\begin{theorem}
For every $\mu=\frac{n-1}{2}, \ n>2$, $\lambda\geq 0$ and $a\geq 0$ we have
\formula{
G^{\lambda}(x,y)\stackrel{\lambda,n}{\approx} \left(\frac{2x_ny_n}{|x-y|^{2}}\right)^{\mu-1/2}\left(1\wedge\frac{2(x_n-a)(y_n-a)}{|x-y|^2}\right)\left(1\wedge\frac{2x_ny_n}{|x-y|^2}\right)^{\nu-1/2}\/,
}
where $x=(\tilde{x},x_n), \ y=(\tilde{y},y_n)$, $x_n,y_n>a$ and $\nu=\sqrt{2\lambda+\mu^2}$. 
\end{theorem}
\begin{proof}
The scaling property of (HBM) gives
\formula{
  G^{\lambda}(ax,ay) = {\tilde{G}}^{\lambda}(x,y)\/,\quad x_n,y_n>1
}
where ${\tilde{G}}_{\H^n}^{\lambda}$ denotes the $\lambda$-Green function for the set $D_1$. Consequently, we will further  consider $a=1$ and omit the sign $\,{\tilde{}}\,$.  We recall the recent result giving the sharp estimates of the transition density of the Bessel process killed when leaving the half-line $(a,\infty)$, i.e. it was shown in \cite{BogusMalecki:2014} that for every $\nu\neq 0$ we have
\formula[eq:pt1:estimates]{
p_a^{(-\nu)}(t,x_n,y_n)\stackrel{\nu}{\approx}\left(1\wedge\frac{(x_n-1)(y_n-1)}{t}\right)\left(1\wedge\frac{x_n y_n}{t}\right)^{\nu-1/2}\left(x_ny_n\right)^{\nu-1/2}\frac{1}{\sqrt{t}}\exp{\left(-\frac{(x_n-y_n)^2}{2t}\right)}\/,
}
whenever $t>0$ and $x_n,y_n>1$. Note that in \cite{BogusMalecki:2014} the density was considered with respect to the Lebesgue measure and here the reference measure is $m(dy_n)=y_n^{-2\mu+1}\,dy_n$. Splitting the integral (\ref{eq:GreenBessel}) representing $G_\mu^\lambda(x,y)$ into three parts we get
\formula{
G^{\lambda}(x,y)&= \left(\int_{0}^{(x_n-1)(y_n-1)}+\int_{(x_n-1)(y_n-1)}^{x_ny_n}+\int_{x_ny_n}^\infty\right) \frac{(x_ny_n)^{\mu-\nu}}{(2\pi t)^{(n-1)/2}}\exp\left(-\frac{|\tilde{x}-\tilde{y}|^2}{2t}\right)p_t^{(-\nu)}(x_n,y_n)\,dt\\
& \approx J_1(x,y)+J_2(x,y)+J_3(x,y)\/,
}
where, by substituting $w=|x-y|^2/(2t)$ we get
\formula[eq:J1:formula]{
\nonumber
J_1(x,y) & = (x_ny_n)^{\mu-1/2}\int_0^{(x_n-1)(y_n-1)}t^{-n/2}\exp\left(-\frac{|x-y|^2}{2t}\right)dt \\
&= \left(\frac{2x_ny_n}{|x-y|^2}\right)^{n/2-1} \int_{\frac{|x-y|^2}{2(x_n-1)(y_n-1)}}^\infty w^{n/2-2}e^{-w}dw
}
and similarly 
\formula[eq:J2:formula]{
J_2(x,y) &= \left(\frac{2x_ny_n}{|x-y|^2}\right)^{n/2-1}\frac{2(x_n-1)(y_n-1)}{|x-y|^2}\int_{\frac{|x-y|^2}{2x_ny_n}}^{\frac{|x-y|^2}{2(x_n-1)(y_n-1)}}w^{n/2-1}e^{-w}dw\/,
}
\formula[eq:J3:formula]{
J_3(x,y) &= \left(\frac{2x_ny_n}{|x-y|^2}\right)^{\mu+\nu-1}\frac{2(x_n-1)(y_n-1)}{|x-y|^2} \int_0^{\frac{|x-y|^2}{2x_ny_n}}w^{\mu+\nu-1}e^{-w}dw\/.
}
Note that if $|x-y|^2<2(x_n-1)(y_n-1)$ then the integral in (\ref{eq:J1:formula}) behaves like a constant and 
\formula{
J_1(x,y) \approx \left(\frac{2x_ny_n}{|x-y|^2}\right)^{n/2-1}\/.
}
Moreover, since $|x-y|^2<2(x_n-1)(y_n-1)<2x_ny_n$ we have
\formula{
J_2(x,y) &\approx\left(\frac{2x_ny_n}{|x-y|^2}\right)^{n/2-1}\frac{2(x_n-1)(y_n-1)}{|x-y|^2}\int_0^{\frac{|x-y|^2}{2(x_n-1)(y_n-1)}}w^{n/2-1}dw\\
&\approx\left(\frac{2x_ny_n}{|x-y|^2}\right)^{n/2-1}\left(\frac{|x-y|^2}{2(x_n-1)(y_n-1)}\right)^{n/2-1} <\left(\frac{2x_ny_n}{|x-y|^2}\right)^{n/2-1}
}
and since
\formula{
J_3(x,y) &< \left(\frac{2x_ny_n}{|x-y|^2}\right)^{\mu+\nu-1}\frac{2(x_n-1)(y_n-1)}{|x-y|^2} \int_0^{\frac{|x-y|^2}{2x_ny_n}}w^{\mu+\nu-1}dw\\
&< \frac{(x_n-1)(y_n-1)}{x_ny_n}\\
&<\left(\frac{2x_ny_n}{|x-y|^2}\right)^{n/2-1}\left(\frac{|x-y|^2}{2(x_n-1)(y_n-1)}\right)^{n/2-1} <\left(\frac{2x_ny_n}{|x-y|^2}\right)^{n/2-1}\/,
}
which ends the proof in this case.

In the second case, i.e when $2(x_n-1)(y_n-1)\leq |x-y|^2\leq x_ny_n$, the integral part of (\ref{eq:J2:formula}) is comparable with a constant and consequently, $J_2(x,y)$ dominates the other parts. To see that note the asymptotic formula for the incomplete gamma function
\formula[eq:igf:infty]{
\int_z^{\infty}w^{\alpha-1}e^{-w}dw \approx z^{\alpha-1}e^{-z}\/,\quad z\to \infty\/.
}
which gives that the integral part of (\ref{eq:J1:formula}) decays exponentially as $|x-y|^2/(2(x_n-1)(y_n-1))$ grows to infinity and consequently $J_1(x,y)/J_2(x,y)$ is bounded from above in the considered region. Moreover, the above given estimates of $J_3(x,y)$ and $J_2(x,y)$ give
\formula{
J_3(x,y) &<\frac{(x_n-1)(y_n-1)}{x_ny_n} \leq 1\approx J_2(x,y)\/.
}

The final part of the proof relates to the case when $2(x_n-1)(y_n-1)\leq |x-y|^2$ and $x_ny_n\leq |x-y|^2$. Then, by the other assumption, the integral part of (\ref{eq:J3:formula}) behaves like a constant and to finish the proof it is enough to show that $J_3(x,y)$ dominates $J_1(x,y)$ and $J_2(x,y)$ in that case. Indeed, using (\ref{eq:igf:infty}) we get
\formula{
  J_2(x,y) &< \left(\frac{2x_ny_n}{|x-y|^2}\right)^{n/2-1}\frac{2(x_n-1)(y_n-1)}{|x-y|^2}\int_{\frac{|x-y|^2}{2x_ny_n}}^\infty w^{n/2-1}e^{-w}dw\\
  &\approx \left(\frac{2x_ny_n}{|x-y|^2}\right)^{n-2}\frac{2(x_n-1)(y_n-1)}{|x-y|^2} \left(\frac{|x-y|^2}{2x_ny_n}\right)^{n-2}\exp\left(-\frac{|x-y|^2}{2x_ny_n}\right)\\
	&\leq \left(\frac{2x_ny_n}{|x-y|^2}\right)^{\mu+\nu-1}\frac{2(x_n-1)(y_n-1)}{|x-y|^2} \left(\frac{|x-y|^2}{2x_ny_n}\right)^{n-2}\exp\left(-\frac{|x-y|^2}{2x_ny_n}\right)
  &< c_1(n) J_3(x,y)\/,
}
where $c_1(n)=(\int_{0}^{1/2}w^{n-2}e^{-w}dw)^{-1} \cdot \sup_{x\geq 1/2}x^{n-2}e^{-x}$. Finally, 
\formula{
  J_1(x,y) &<\left(\frac{2x_ny_n}{|x-y|^2}\right)^{n/2-1}\left(\frac{|x-y|^2}{2(x_n-1)(y_n-1)}\right)^{n/2-2}\exp\left(-\frac{|x-y|^2}{2(x_n-1)(y_n-1)}\right)\\
  & < \left(\frac{2x_ny_n}{|x-y|^2}\right)^{n-2}\frac{2(x_n-1)(y_n-1)}{|x-y|^2} \left(\frac{|x-y|^2}{2(x_n-1)(y_n-1)}\right)^{n-2}\exp\left(-\frac{|x-y|^2}{2(x_n-1)(y_n-1)}\right)\\
	  & < \left(\frac{2x_ny_n}{|x-y|^2}\right)^{\mu+\nu-1}\frac{2(x_n-1)(y_n-1)}{|x-y|^2} \left(\frac{|x-y|^2}{2(x_n-1)(y_n-1)}\right)^{n-2}\exp\left(-\frac{|x-y|^2}{2(x_n-1)(y_n-1)}\right)\\
  &<c_1(n)J_3(x,y)\/.
}
This ends the proof.
\end{proof}

The above-given result can be stated in terms of the hyperbolic distance in the following way.
\begin{corollary}
   For every $n>2$ and $\lambda>0$ we have
		\formula{
	   G^{\lambda}(x,y)\stackrel{\lambda,n}{\approx} \frac{1}{\sinh^{2\mu-1}(d_{\H^n}(x,y)/2)\cosh^{\nu+1/2}(d_{\H^n}(x,y))}\left(1\wedge \frac{(1\wedge \delta_a(x))(1\wedge \delta_a(y))}{1\wedge d_{\H^n}^2(x,y)}\right)\/,\quad x,y\in D\/,
	}
	where $\delta_a(x)$ denotes the (hyperbolic) distance of $x$ to the boundary of set $D$, or equivalently 
	\formula{
	G^{\lambda}(x,y)\stackrel{\lambda,n}{\approx}\left(1\wedge \frac{(1\wedge \delta_a(x))(1\wedge \delta_a(y))}{1\wedge d_{\H^n}^2(x,y)}\right)\,U^\lambda(x,y)\/,\quad x,y\in D\/.
	  }
\end{corollary}

\begin{proof}
 First observe that by \eqref{eq:distance:formula} we simply have
\formula{
   1\wedge \frac{2x_ny_n}{|x-y|^2} \approx \left(1+\frac{|x-y|^2}{2x_ny_n}\right)^{-1} = \cosh^{-1}(d_{\H^n}(x,y))\/.
}
Moreover, \eqref{eq:distance:formula} also implies that
\formula{
   \frac{2x_ny_y}{|x-y|^2} = \left(\cosh(d_{\H^n}(x,y))-1\right)^{-1} = \frac12 \sinh^{-2}(d_{\H^n}(x,y)/2)\/.
}
Since the lines perpendicular to the boundary of the set $D$ are geodesis of the hyperbolic space, the distance of $x=(\tilde{x},x_n)$ to the boundary of $D$ is realized as a distance between $x$ and $(\tilde{x},a)$. Thus we have
\formula{
   \cosh\delta_a(x) = 1+\frac{(x_n-a)^2}{2x_na} = \frac12\left(\frac{x_n}{a}+\frac{a}{x_n}\right)\/, \quad \delta_a(x) = \ln(x_n/a)\/.  
}
Consequently
\formula{
   \frac{x_n}{a}-1 = e^{\delta_a(x)}-1\approx \sinh \delta_a(x)\/,\quad \frac{x_n}{a} = e^{\delta_a(x)}\approx \cosh\delta_a(x)\/.
}
Using this we can write
\formula{
   \frac{(x_n-a)(y_n-a)}{|x-y|^2} = \frac{(x_n/a-1)(y_n/a-1)}{2x_ny_n/a^2}\frac{2x_ny_n}{|x-y|^2}\approx \frac{\sinh\delta_a(x)\,\sinh\delta_a(y)}{\cosh\delta_a(x)\cosh\delta_a(y)}\frac{1}{2\sinh^2(d_{\H^n}(x,y)/2)}\/.
}
To finish the proof note that $\tanh\delta_a(x)\approx 1\wedge \delta_a(x)$ and 
\formula{
   \left(1\wedge \frac{(1\wedge\delta_a(x))(1\wedge\delta_a(y))}{\sinh^2(d_{\H^n}(x,y)/2)}\right)\cosh d_{\H^n(x,y)}\approx (1\wedge\delta_a(x))(1\wedge\delta_a(y)) 
}
whenever $d_{\H^n}(x,y)\geq 1$ and 
\formula{
  \left(1\wedge \frac{(1\wedge\delta_a(x))(1\wedge\delta_a(y))}{\sinh^2(d_{\H^n}(x,y)/2)}\right)\cosh d_{\H^n(x,y)}\approx 1\wedge \frac{(1\wedge\delta_a(x))(1\wedge\delta_a(y))}{d_{\H^n}^2(x,y)}
}
if only $d_{\H^n}(x,y)< 1$. Collecting all together we get the result.
\end{proof}
We end with the following conjecture stating that the form of the above-given estimates should be the same for more general smooth subsets of $\H^n$. Obviously, the conjecture is true if we assume that the considered set $D$ is bounded in hyperbolic metric. It follows from the fact that the potential theory on bounded subsets of $\H^n$ is comparable to Euclidean potential theory, since the Laplace-Beltrami operator is strongly elliptic operator on such domains.
\begin{conjecture}
   For every $C^{1,1}$ domain $D\subset \H^n$ we have
	\formula{
	   G_D(x,y) \stackrel{n,D}{\approx} \left(1\wedge \frac{(1\wedge \delta_a(x))(1\wedge \delta_a(y))}{1\wedge d_{\H^n}^2(x,y)}\right)\,U(x,y)\/,\quad x,y\in D\/.
	}
\end{conjecture} 


\bibliography{bibliography}

\begin{thebibliography}{10}

\bibitem{AbramowitzStegun:1972}
M.~Abramowitz and I.~A. Stegun.
\newblock {\em Handbook of Mathematical Functions with Formulas, Graphs, and
  Mathematical Tables}.
\newblock Dover, New York, 9th edition, 1972.

\bibitem{BogusMalecki:2015}
K.~Bogus and J.~Ma\l{}ecki.
\newblock Heat kernel estimates for the {Bessel} differential operator in
  half-line.
\newblock {\em preprint 2015}, arXiv:1501.02618.

\bibitem{BogusMalecki:2014}
K.~Bogus and J.~Ma\l{}ecki.
\newblock Sharp estimates of transition probability density for {Bessel}
  process in half-line.
\newblock {\em Potential Anal.}, to appear (DOI: 10.1007/s11118-015-9461-x).

\bibitem{BGS:2007}
T.~Byczkowski, P.~Graczyk, and A.~Stos.
\newblock Poisson kernels of half-spaces in real hyperbolic spaces.
\newblock {\em Rev. Mat. Iberoamericana}, 23(1):85--126, 2007.

\bibitem{BMR3:2013}
T.~Byczkowski, J.~Ma{\l}ecki, and M.~Ryznar.
\newblock Hitting times of {Bessel} processes.
\newblock {\em Potential Anal.}, 38:753--786, 2013.

\bibitem{BMZ:2010}
T.~Byczkowski, J.~Ma{\l}ecki, and T.~\.Z{}ak.
\newblock Feynman-{Kac} formula, $\lambda$-{Poisson} kernels and
  $\lambda$-{Green} functions of half-spaces and balls in hyperbolic spaces.
\newblock {\em Colloq. Math.}, 118:201--222, 2010.

\bibitem{BR:2006}
T.~Byczkowski and M.~Ryznar.
\newblock Hitting distibution of geometric {Brownian} motion.
\newblock {\em Studia Math.}, 173(1):19--38, 2006.

\bibitem{DonatiMartinYor:1997}
C.~Donati-Martin and M.~Yor.
\newblock Some {Brownian} functionals and their laws.
\newblock {\em Ann. Prob.}, 25:1011--1056, 1997.

\bibitem{Erdelyi:1953:volII}
Erdelyi et~al.
\newblock {\em Higher Transcendental Functions}, volume~II.
\newblock McGraw-Hill, New York, 1953.

\bibitem{Erdelyi:1954}
Erdelyi et~al.
\newblock {\em Tables of integral transforms}, volume I, II.
\newblock McGraw-Hill, New York, 1954.

\bibitem{GoingYor:2003}
A.~Going-Jaeschke and M.~Yor.
\newblock Survey and some generalizations of {Bessel} processes.
\newblock {\em Bernoulli}, 9:313--350, 2003.

\bibitem{IW}
N.~Ikeda and S.~Watanabe.
\newblock {\em Stochastic Differential Equations and Diffusion Processes}.
\newblock North-Holland, 1981.

\bibitem{JakubowskiWiesniewolski:2013a}
J.~Jakubowski and M.~Wi\'s{}niewolski.
\newblock On hyperbolic {Bessel} processes and beyond.
\newblock {\em Bernoulli}, 19(5B):2437--2454, 2013.

\bibitem{JakubowskiWiesniewolski:2013b}
J.~Jakubowski and M.~Wi\'s{}niewolski.
\newblock On some {Brownian} functionals and their applications to moments in
  the lognormal stochastic volatility model.
\newblock {\em Studia Math.}, 219:201--224, 2013.

\bibitem{MaleckiSerafin:2012}
J.~Ma\l{}ecki and G.~Serafin.
\newblock Hitting hyperbolic half-space.
\newblock {\em Demonstratio Mathematica}, 45(2):337--360, 2012.

\bibitem{MatsumotoYor:2005a}
H.~Matsumoto and M.~Yor.
\newblock Exponential functionals of {Brownian} motion, {I}: {Probability} laws
  at fixed time.
\newblock {\em Probability Surveys}, 2:312--347, 2005.

\bibitem{RevuzYor:2005}
D.~Revuz and M.~Yor.
\newblock {\em Continuous Martingales and Brownian Motion}.
\newblock Springer, New York, 1999.

\bibitem{Serafin:2014}
G.~Serafin.
\newblock Potential theory of hyperbolic {Brownian} motion in tube domains.
\newblock {\em Colloq. Math.}, 135:27--52, 2014.

\bibitem{Zak:2007}
T.~\.Z{}ak.
\newblock Poisson kernel and {Green} function of balls for complex hyperbolic
  {Brownian} motion.
\newblock {\em Studia Math.}, 183:161--193, 2007.

\end{thebibliography}
\bibliographystyle{plain}

\end{document}